\documentclass[12pt]{amsart}
\usepackage{amsmath}
\usepackage{amssymb}
\usepackage{mathtools,nccmath,textcomp}
\usepackage{tikz}
\usepackage[all]{xy}
\usepackage{longtable}
\usepackage{url}
\usepackage{textcomp}
\allowdisplaybreaks[1]
\makeatletter
\@namedef{subjclassname@2020}{%
  \textup{2020} Mathematics Subject Classification}
\makeatother

\newtheorem{theorem}{Theorem}
\newtheorem{lemma}[theorem]{Lemma}

\newtheorem{corollary}[theorem]{Corollary}

\theoremstyle{definition}

\newtheorem{example}[theorem]{Example}

\theoremstyle{remark}
\newtheorem{remark}[theorem]{Remark}
\newtheorem{problem}[theorem]{Problem}
\numberwithin{equation}{section}

\DeclareMathAlphabet{\matheur}{U}{eur}{m}{n}

\newcommand{\R}{\mathbb{R}}
\newcommand{\Q}{\mathbb{Q}}
\newcommand{\N}{\mathbb{N}}
\newcommand{\Z}{\mathbb{Z}}

\newcommand{\lcm}{\mathrm{lcm}}

\newcommand{\ord}{\mathrm{ord}}

\renewcommand\d{{\mathrm d}}

\newmuskip\pFqskip
\mathchardef\pFcomma=\mathcode`, % keep a copy of the comma

\renewcommand{\d}{\mathrm d}

\begin{document}
%~
\title[A dynamical system proof of Niven's theorem and its extensions]{A dynamical system proof of Niven's theorem and its extensions}

\author{Chatchawan Panraksa}
\address{Applied Mathematics Program, Mahidol University International College, Nakhon Pathom, Thailand 73170} \email{chatchawan.pan@mahidol.edu}

\author{Detchat Samart}
\address{Department of Mathematics, Faculty of Science, Burapha University, Chonburi, Thailand 20131} \email{petesamart@gmail.com}

\author{Songpon Sriwongsa}
\address{Department of Mathematics, Faculty of Science, King Mongkut's University of Technology Thonburi, Bangkok, Thailand 10140} \email{songpon.sri@kmutt.ac.th}

%\thanks{}


\date{\today}

\maketitle

\begin{abstract}
Niven's theorem asserts that $\{\cos(r\pi) \mid r\in \Q\}\cap \Q=\{0,\pm 1,\pm 1/2\}.$ In this paper, we use elementary techniques and results from arithmetic dynamics to obtain an algorithm for classifying all values in the set $\{\cos(r\pi) \mid r\in \Q\}\cap K$, where $K$ is an arbitrary number field.
\end{abstract}

\section{Introduction}\label{S:intro}
Finding exact values of trigonometric functions is a classical problem. For the sake of simplicity, we will mainly focus on cosine; analogous results for other trigonometric functions can be deduced using simple identities like 
\begin{align*}
\sin(\theta)&= \cos\left(\frac{\pi}{2}-\theta\right),\\\tan^2(\theta)&=\frac{1-\cos(2\theta)}{1+\cos(2\theta)}.
\end{align*}
 Common trigonometric values which are usually covered in introductory trigonometry lessons include 
\[\cos(0)=1,\quad \cos\left(\frac{\pi}{6}\right)=\frac{\sqrt{3}}{2},\quad \cos\left(\frac{\pi}{4}\right)=\frac{\sqrt{2}}{2},\quad \cos\left(\frac{\pi}{3}\right)=\frac{1}{2},\quad \cos\left(\frac{\pi}{2}\right)=0.\]
There are many other exact values of cosine which are not as common as the above values such as 
\[\cos\left(\frac{\pi}{12}\right)=\frac{\sqrt{6}+\sqrt{2}}{4},\quad \cos\left(\frac{\pi}{5}\right)=\frac{1+\sqrt{5}}{4}.\]
Observe that all of these values are expressible using only arithmetic operations and square roots, so they are algebraic numbers. In fact, by suitably applying de Moivre's formula, it can be seen that for any $r\in \Q$, $\cos(r\pi)$ is an algebraic number. (See, for example, \cite{Jahnel}). On the other hand, if $r$ is an {\it irrational algebraic number}, then it follows from the Gelfond-Schneider theorem \cite[Thm.~10.1]{Niven} that $\cos(r\pi)$ is transcendental. It is therefore an interesting problem to explicitly determine values of $\cos(r\pi)$ when $r$ is rational. Since algebraic number fields; i.e., finite field extensions of $\Q$, constitute the field of algebraic numbers, this problem can also be rephrased as follows.
\begin{problem}\label{Pr:main}
Given a number field $K$, find all elements in $K$ which are values of cosine at a rational multiple of $\pi$.
\end{problem}
By the Abel-Ruffini theorem, if $K$ is a number field of degree higher than four, it might not be possible to write an element of $K$ in a closed form. By ``explicitly determining'' an algebraic number $\alpha$, we generally refer to finding the minimal polynomial of $\alpha$. It should be remarked that Problem~\ref{Pr:main} is neither new nor open; its complete solution, in some sense, has been known for some time now. (See a remark after Theorem~\ref{T:Lehmer} below.) A prime example is the following theorem \cite[Cor.~3.12]{Niven}, which corresponds to $K=\Q$.
\begin{theorem}[Niven's theorem]\label{T:Niven}
If $r$ and $\cos(r\pi)$ are both rational, then $\cos(r\pi) \in \{0,\pm 1, \pm 1/2\}$.
\end{theorem}
Elementary proofs of Theorem~\ref{T:Niven} are given in \cite{Jahnel,PV2021b,PV2021a}. The next result can be seen as an extension of Niven's theorem to quadratic number fields.
\begin{theorem}\label{T:quadratic}
Let $r\in \Q$. If $\cos(r\pi)$ is a quadratic irrationality, then 
\[\cos(r\pi)\in \left\{\pm \frac{\sqrt{2}}{2},\pm \frac{\sqrt{3}}{2}, \frac{\pm1\pm\sqrt{5}}{4}\right\}.\]
\end{theorem}
Using values of cosine listed at the beginning of this section, one sees that the quadratic irrational values of $\cos(r\pi)$, with $r$ rational, correspond to $r\in\{\pm 1/4,\pm 1/6, \pm 1/5, \pm 2/5\}$. Jahnel \cite{Jahnel} proved Theorem~\ref{T:quadratic} using standard tools from algebraic number theory such as prime ideal decomposition and general forms of quadratic integers. We refer the reader to recent work of the second-named author \cite{Samart} for an alternative proof of this theorem which relies purely on basic notions in elementary number theory. For number fields of higher degree, we have the following general result, which is originally due to Lehmer (\cite{Lehmer},\cite[Thm.~3.9]{Niven}).
\begin{theorem}\label{T:Lehmer}
Let $m,n\in \Z$, with $n>2$, be relatively prime. Then $\cos(2\pi m/n)$ is an algebraic number of degree $\varphi(n)/2$, where $\varphi(n)$ is the Euler's totient function. 
\end{theorem}
With the help of this result, one can resolve Problem~\ref{Pr:main} for a number field $K$ of degree $D>1$ by finding all\footnote{It is known that $\varphi(n)\ge \sqrt{n/2}$ for all $n\in\N$. Hence, for any fixed $d\in \N$, there can be at most finitely many $n$ such that $\varphi(n)\mid 2D$.} $n\in \N$ for which $\varphi(n)\mid 2D$ and determining all distinct values among $\cos(2\pi m/n)$, where $m\in \{1,2,\ldots,n\}$ and $(m,n)=1$, which belong to $K$.

Proofs of main results in \cite{Jahnel},\cite{PV2021b},\cite{PV2021a}, and \cite{Samart} make use of the double-angle formula
\begin{equation}\label{E:cos}
\cos(2\theta)=2\cos^2(\theta)-1.
\end{equation}
If we define $F:\R\rightarrow \R$ by $F(x)=2\cos(x)$, then it is obvious from \eqref{E:cos} that $F(x)$ satisfies the functional equation 
\[F(2x)=(F(x))^2-2,\]
from which one can deduce that, for any nonnegative integer $k$,
\begin{equation}\label{E:iterate}
F(2^k x)=f^{(k)}(F(x)),
\end{equation}
where $f(x)=x^2-2$ and $f^{(k)}$ denotes the $k$-fold composition of $f$ with itself. By periodicity of cosine, the set $\{F(2^kr\pi)\mid k\ge 0\}$ is finite for any $r\in \Q$. Niven' s theorem can be then proven easily by iteratively applying $f(x)$ to a rational value of $F(r\pi)$ and using the fact that $F(x)\in [-2,2]$ for any $x\in \R$. Observe that this argument has dynamical flavor as it involves iteration of the rational map $f(x)=x^2-2$. Therefore, we have an intuitive conviction that these proofs can be rewritten in the language of dynamical systems. The main purpose of this article is to present a systematic approach to solving Problem~\ref{Pr:main} using ideas from arithmetic dynamics. Our first main result is the following theorem.
\begin{theorem}\label{T:CK}
Let $K$ be a number field and let $C(K):=\{2\cos(r\pi)\mid r\in \mathbb{Q}\}\cap K$. Define $f:K\rightarrow K$ by $f(x)=x^2-2$. Then we have
\[C(K) = \mathrm{PrePer}(f,K),\]
where $\mathrm{PrePer}(f,K)$ denotes the set of preperiodic points of $f$ in $K$.
\end{theorem}
Problem~\ref{Pr:main} then boils down to finding the preperiodic points of $f(x)=x^2-2$ over $K$. By Northcott's theorem \cite{Northcott}, we have that the set $\mathrm{PrePer}(f,K)$ is finite for any number field $K$. Although determining all elements of this set in general is not an easy task, there exists a procedure which allows us to compute them in a finite number of steps. More precisely, we shall prove the following result.
\begin{theorem}\label{T:deg}
Let $K$ be a number field of degree $D$ and let $\alpha\in K$ be a periodic point of $f(x)=x^2-2$ with minimal period $n$. Then we have $n\mid D$. In particular, $f^{(D)}(\alpha)=\alpha.$
\end{theorem}
Choosing $K=\Q$ in Theorem~\ref{T:deg}, one sees that the rational periodic points of $f$ must be its fixed points. In general, the periodic points of $f$ which belong to some number field of degree $D$ are exactly the zeros of irreducible factors of the polynomial $f^{(D)}(x)-x$ whose degrees do not exceed $D$. Since all preperiodic points can be obtained from the periodic points via the inverse mappings of $f$, we can systematically compute $\mathrm{PrePer}(f,K)$ using this result. Detailed computations over number fields of degree up to five will be illustrated in Section~\ref{S:ex}. We prove Theorem~\ref{T:CK}, Theorem~\ref{T:deg}, and some other related results in Section~\ref{S:main}. In the next section, we review basic definitions and notions in arithmetic dynamics. Especially, we invoke some known results about dynamical properties of the map $f(x)=x^2-2$, which are crucial for the proof of Theorem~\ref{T:deg}. 

\section{Arithmetic dynamics of the map $f(x)=x^2-2$}\label{S:nab}
We first briefly recall some basic concepts frequently used in the study of arithmetic dynamics. For any map $g$ from a set $S$ to itself and $k\in\mathbb{N}$, we define $g^{(k)}(x):=\underbrace{g(g(\cdots g}_{k \textrm{ times}}(x))).$ We say that a point $P\in S$ is {\it periodic} with respect to $g$ if $g^{(n)}(P)=P$ for some $n\in\mathbb{N}$ and we call the smallest such $n$ the {\it minimal period} for $P$. The point $P$ is said to be {\it preperiodic} if $g^{(m)}(P)$ is periodic for some $m\in \mathbb{N}$, which is equivalent to that the {\it forward orbit} 
\[\mathcal{O}_g(P):=\{g^{(k)}(P)\mid k\ge 0\}\]
is finite. The set of preperiodic points of $g$ (in $S$) is denoted by $\mathrm{PrePer}(g,S)$.\\

If a field $K$ does not have characteristic 2, then any quadratic polynomial $g(x)=Ax^2+Bx+C\in K[x]$ with $A\neq 0$ can be transformed into $f_c(x)=x^2+c$ by changing  variables
$$f_c(x)=\varphi^{-1}\circ g\circ \varphi(x),$$
where $\varphi(x)=\dfrac{2x-B}{2A}$ and $c=\dfrac{B}{2}-\dfrac{B^2}{4}+AC.$ The study of an orbit of a rational function plays a central role in arithmetic dynamics. For example, given a number field $K$, the Morton-Silverman Uniform Boundedness Conjecture \cite{MS} asks if we can have an upper bound for the number of preperiodic points of $f(x)\in K[x]$ depending only on $[K:\mathbb{Q}]$ and $\deg f$. It is still open for the family of quadratic polynomials with $K=\mathbb{Q}$. For periods 1, 2, and 3 of $f_c(x)\in \mathbb{Q}[x]$, we can explicitly describe the relationship between the rational preperiodic points and the parameter $c$ (see \cite{WR}). In addition, it is known that $f_c(x)$ has no rational periodic points of minimal periods $4, 5$, and 6 (conditionally on a version of the Birch and Swinnerton-Dyer conjecture) (see \cite{Morton},\cite{FPS}, and \cite{Stoll}). For polynomials with integral coefficients, by using a simple divisibility argument, it can be shown that all preperiodic points have periods at most 2. For a number field $K$ and $c\in \mathcal{O}_K$, a result in \cite{CP} implies that the set of preperiodic points of $f_c(x)$ is uniformly bounded depending only on $D=[K:\mathbb{Q}]$ (see also \cite{Zhang}). However, it is still interesting to find an explicit bound for the number of preperiodic points of an integral quadratic polynomial. In this paper, we describes the preperiodic points of $f_{-2}(x)=x^2-2$ over $K$, where the base field $K$ varies. 

An important tool to study preperiodic points of a polynomial is its {\it dynatomic polynomial}.
A dynatomic polynomial is a polynomial that encodes information about the orbits of a point under iteration of a polynomial map $f(x)$. Given a polynomial $f(x)$ and a positive integer $n$, the $n^{th}$ {\it dynatomic polynomial} of $f$, denoted by $\Phi_{n,f}(x)$, is the polynomial whose roots are the points that remain fixed under iteration of $f(x)$ for exactly $n$ steps. These fixed points are also called the {\it formal $n$-periodic points}, and are solutions of the equation $f^{(n)}(x) = x$. As an analogue of cyclotomic polynomials, the $n^{th}$ dynatomic polynomial can be computed using the following formula:
$$\Phi_{n,f}(x)= \prod_{d\mid n}\left(f^{(d)}(x)-x\right)^{\mu \left(\frac{n}{d}\right)},$$
where $\mu$ is the M\"{o}bius function defined by $\mu(1)=1$ and
\[
\mu(p_1^{i_1}p_2^{i_2}\cdots p_k^{i_k})=
\begin{cases}
        (-1)^k & \text{if } i_j=1 \text{ for all } j\in \{1,2,\dots,k\}, \\
        0 & \text{if } i_j\geq 2 \text{ for some } j\in \{1,2,\dots,k\}.
    \end{cases}
\]
For example, the first few dynatomic polynomials of $f_c(x)$ are
\begin{align*}
\Phi_{1,f_c}(x)&= x^2-x+c,\\
\Phi_{2,f_c}(x)&= x^2+x+(c+1),\\
\Phi_{3,f_c}(x)&= x^6+x^5+(3c+1)x^4+(2c+1)x^3+(3c^2+3c+1)x^2+(c^2+2c+1)x\\
&\quad +(c^3+2c^2+c+1).
\end{align*}
The dynatomic polynomials are important in the study of arithmetic dynamics, in particular in the study of the arithmetic and geometric properties of the orbits of points under iteration of polynomial maps. For more details of the dynatomic polynomials, see \cite{bousch1992quelques,Morton2,Silverman}.

Vivaldi and Hatjispyros \cite[\S 5.2]{VH} explicitly described the $n^{th}$ dynatomic polynomial of $f(x)=x^2-2$ in the following theorem.
\begin{theorem}\label{T:phi}
Let $f(x)=x^2-2$ and define $\Psi_m(x+x^{-1})=\Phi_m(x)x^{-\varphi(m)/2}$, where $\Phi_m(x)$ is the $m^{\text{th}}$ cyclotomic polynomial. Then we have
\begin{equation*}
\Phi_{n,f}(x)=\prod_{d\mid n}\left(\prod_{d_1\mid 2^d-1}\Psi_{d_1}(x)\prod_{d_2\mid 2^d+1}\Psi_{d_2}(x)\right)^{\mu\left(\frac{n}{d}\right)}.
\end{equation*}

\end{theorem}
\begin{remark}
It is known that $\Psi^2_{1}(x)=x-2,\Psi^2_{2}(x)=x+2,$ and, for $m>2$, $\Psi_{m}(x)$ is an irreducible polynomial with integer coefficients \cite[Lem.~3.8]{Niven}. In fact, it can be seen from a proof of Theorem~\ref{T:Lehmer} that $\Psi_{m}(x)$ is the minimal polynomial of the algebraic integer $2\cos(2\pi/m)$, so $\deg \Psi_{m}=\varphi(m)/2$. 
\end{remark}
\section{Proofs of main results}\label{S:main}
In this section, we prove our main results, namely Theorem~\ref{T:CK} and Theorem~\ref{T:deg}. 
\begin{proof}[Proof of Theorem~\ref{T:CK}]
Let $\alpha=m\pi/n$, where $m,n\in \mathbb{Z}$ with $n>0$ and let 
\[X_n:=\{2\cos(j\pi/n) \mid j\in \mathbb{Z}\}.\]
Suppose that $\gamma:=2\cos(\alpha)\in K$. Then we have
\[f(\gamma)=4\cos^2(\alpha)-2=2\cos(2\alpha)\in X_n.\]
It follows that $\mathcal{O}_f(\gamma)=\{2\cos(2^k\alpha)\mid k\ge 0\}\subseteq X_n.$ By periodicity of cosine, the set $X_n$ is finite, so $\mathcal{O}_f(\gamma)$ must also be finite. Hence $\gamma\in \mathrm{PrePer}(f,K).$

Conversely, let $\delta\in \mathrm{PrePer}(f,K).$ Then there exists $k\in\mathbb{N}$ for which $\beta:=f^{(k)}(\delta)$ is periodic. We first show that $|\delta|\le 2$. Assume to the contrary that $|\delta|>2$. Then it is easy to see that $f^{(j)}(\delta)>2$ for all $j\ge 1$. Moreover, we have 
\[f^{(j+1)}(\delta)-f^{(j)}(\delta)=(f^{(j)}(\delta)-2)(f^{(j)}(\delta)+1)>0,\]
so the sequence $\{f^{(j)}(\delta)\}_{j=1}^\infty$ is strictly increasing. Hence $f^{(j)}(\delta)$ is not periodic for any $j\in \mathbb{N}$, a contradiction. Therefore, we may conclude by continuity of cosine that $\delta=2\cos(r\pi)$ for some $r\in\mathbb{R}$. Since $\beta$ is periodic, there exists $l\in \mathbb{N}$ such that
\[2\cos(2^kr\pi)=f^{(k)}(\delta)=\beta=f^{(l)}(\beta)=f^{(k+l)}(\delta)=2\cos(2^{k+l}r\pi).\]
Thus $2^kr\pi =\pm 2^{k+l}r\pi+2t\pi$ for some $t\in\mathbb{Z}.$ It is clear from this equation that $r$ is rational, so $\delta\in C(K)$ as desired.
\end{proof}
\begin{remark} In fact, Theorem~\ref{T:CK} holds for any \textit{Chebyshev polynomial} $P_n$ defined by 
\begin{align*}
&P_1(x)=x, \quad P_2(x)=x^2-2, \text{ and }\\
&P_{m+1}(x)=xP_m(x)-P_{m-1}(x), \text{ for } m\ge 2.
\end{align*}
A proof of this more general result can be found in \cite[Prop.~2.2.2]{IT}. Since our proof is quite simple and should be accessible to more general audience, we decide to include it here rather than solely referring to the result above.
\end{remark}
It is a well-known fact that if $r$ is rational, then $2\cos(r\pi)$ is an algebraic integer. Proofs of this result using theory of algebraic numbers can be found in \cite[Thm.~3.9]{Niven} and \cite{Jahnel}. On the other hand, we shall apply Theorem~\ref{T:CK} to prove this fact, without resorting to any advanced machinery in algebraic number theory.
\begin{corollary}
Let $K$ be a number field. Then $C(K)\subseteq \mathcal{O}_K$, where $\mathcal{O}_K$ denotes the ring of algebraic integers of $K$.
\end{corollary}
\begin{proof}
Let $\gamma\in C(K).$ Then by Theorem~\ref{T:CK} we have that there exists $m\in \mathbb{N}$ such that $f^{(m)}(\gamma)$ is a periodic point, where $f(x)=x^2-2.$ Hence there exists $l\in \mathbb{N}$ such that 
\[f^{(l+m)}(\gamma)=f^{(l)}(f^{(m)}(\gamma))=f^{(m)}(\gamma).\]
Let $h(x)=f^{(l+m)}(x)-f^{(m)}(x).$ Then it is obvious that $h(x)\in \Z[x]$ is monic and annihilates $\gamma$. Therefore, we have that $\gamma\in \mathcal{O}_K.$ 
\end{proof}
Using Theorem~\ref{T:Lehmer} and Theorem~\ref{T:CK}, one can easily obtain an explicit upper bound for the order of $\mathrm{PrePer}(f,K)$ in terms of $[K:\Q]$, in line with Northcott's theorem.
\begin{corollary}\label{Co:bound}
Let $K$ be a number field of degree $D$ and $f(x)=x^2-2$. Then we have
\[|\mathrm{PrePer}(f,K)| \le \sum_{n=1}^{8D^2}\varphi(n).\]
\end{corollary}
\begin{proof}
By Theorem~\ref{T:Lehmer} and Theorem~\ref{T:CK}, we have
\[\mathrm{PrePer}(f,K)\subseteq \left\{2\cos\left(2\pi \frac{m}{n}\right): n\in\N, \varphi(n)\mid 2D, 1\le m<n, (m,n)=1\right\}.\]
If $\varphi(n)\mid 2D$, then by a trivial lower bound for $\varphi(n)$, we have $\sqrt{n/2} \le \varphi(n) \le 2D$, implying $n\le 8D^2$. Hence the cardinality of the set on the right-hand side is at most $\displaystyle\sum_{n=1}^{8D^2}\varphi(n)$.
\end{proof}
\begin{remark}
 The summation in Corollary~\ref{Co:bound} can be written as a value of the \textit{totient summatory function} $\Phi(m):=\displaystyle\sum_{n=1}^{m}\varphi(n)$, which satisfies an asymptotic formula
 \[\Phi(m)\sim \frac{3m^2}{\pi^2}+O(m \log m).\] One can modify a proof of this formula to obtain a simpler upper bound for $|\mathrm{PrePer}(f,K)|$. Note, however, that this bound is very far from optimal. One way to improve it is to use a stronger lower bound for $\varphi(n)$. For instance, it is known from \cite[Thm.~8.8.7]{BS} that for $n>2$
 \[\varphi(n)>\frac{n}{e^\gamma \log\log n+\frac{3}{\log\log n}},\]
 where $\gamma$ is the Euler's constant.
\end{remark}
To prove Theorem~\ref{T:deg}, we need the following auxiliary results about divisors of $2^l\pm 1$, where $l$ is a prime power. Recall that for a prime $p$ and a nonzero integer $s$ the {\it $p$-adic valuation of $s$}, denoted by $v_p(s)$, is the exponent of the highest power of $p$ that divides $s$. 
\begin{lemma}\label{L:divisor}
Let $k\in\N$.
\begin{itemize}
\item[(i)] If $k>1$ and $q$ is a prime divisor of $2^{2^k}+1$, then $q=2^{k+2}m+1$ for some $m\in \N$.
\item[(ii)] For any odd prime $p$, if $q$ is a prime divisor of $2^{p^k}-1$ (resp. $2^{p^k}+1$) and $q \nmid 2^{p^{k-1}}-1$ (resp. $q \nmid 2^{p^{k-1}}+1$), then $q=2p^km+1$ for some $m\in \N$.
\item[(iii)] For any odd prime $p$,
\begin{align}
v_3(2^{p^k}+1)&=\begin{cases}
k+1 &\text{ if } p=3,\\
1   &\text{ if } p>3,
\end{cases}\label{E:v3p}\\
v_3(2^{p^k}-1)&=0. \label{E:v3m}
\end{align}
\end{itemize}
\end{lemma}
\begin{proof}
(i) This assertion is known as the Euler-Lucas theorem \cite[Thm.~1.3.5]{CP}, which gives an explicit form of the prime divisors of Fermat numbers. 

(ii) Recall that for a positive integer $n$ and an integer $a$ which is relatively prime to $n$, the {\it order of $a$ modulo $n$}, denoted by $\ord_n a$, is the smallest positive integer $r$ such that 
\[a^r \equiv 1 \pmod n.\]
Indeed, for any $s\in \N$, if $a^s\equiv 1 \pmod n$, then $\ord_n a \mid s$. Let $p$ be an odd prime and let $q$ be a prime divisor of $2^{p^k}-1$, where $q\nmid 2^{p^{k-1}}-1$. Then it follows immediately that $\ord_q 2=p^k.$ Since $q$ is odd, we have by Fermat's little theorem that $2^{q-1}\equiv 1 \pmod q$, implying $p^k \mid q-1$. Moreover, since $2\mid q-1$ and $q>1$, we have $q=2p^km+1$ for some integer $m\ge 1$, as desired.

Next, assume that $q\mid 2^{p^k}+1$, but $q\nmid 2^{p^{k-1}}+1$. Obviously, $3\mid 2^l+1$ for every $l\in \N$, so $q>3$. Observe that 
\[2^{2p^k}=\left(2^{p^k}\right)^2\equiv (-1)^2=1 \pmod q,\]
so $\ord_q 2\mid 2p^k$. Since $q\ne 3$, one sees that $\ord_q 2\ne 2$. Moreover, since $2^{p^{k-1}}\not\equiv \pm 1 \pmod q$, we have that $2^{2p^{k-1}}\not\equiv 1 \pmod q$. Therefore, we can conclude that $\ord_q 2= 2p^k.$ Again, in consequence of Fermat's little theorem, we can write $q=2p^km+1$ for some $m\in \N$.

(iii) Applying the lifting-the-exponent lemma \cite[Thm.~6.2]{BR}, one sees immediately that \eqref{E:v3p} holds. In addition, since $3\mid 2^{p^k}+1$, we have $2^{p^k}-1 \equiv -2 \mod 3$, which yields \eqref{E:v3m}.
\end{proof}
\begin{lemma} \label{L:ppower}
Let $t=p^k$, where $p$ is prime and $k\in\N$. Then the degree of each irreducible factor of $\Phi_{t,f}(x)$ over $\Q$ is a multiple of $t$.
\end{lemma}
\begin{proof}
We divide the proof into three cases, in accordance with the value of $p$.

{\it Case $p=2$.}

We have from Theorem~\ref{T:phi} that  
\begin{align*}
\Phi_{t,f}(x)&=\prod_{r=0}^k\left(\prod_{d_1\mid 2^{2^r}-1}\Psi_{d_1}(x)\prod_{d_2\mid 2^{2^r}+1}\Psi_{d_2}(x)\right)^{\mu\left(2^{k-r}\right)}\\
&= \prod_{d_1\mid 2^{2^k}-1}\Psi_{d_1}(x)\prod_{d_2\mid 2^{2^k}+1}\Psi_{d_2}(x) \prod_{d_1\mid 2^{2^{k-1}}-1}\Psi_{d_1}(x)^{-1}\prod_{d_2\mid 2^{2^{k-1}}+1}\Psi_{d_2}(x)^{-1}\\
&=\prod_{\substack{c_1\mid 2^{2^{k-1}}-1\\c_2\mid 2^{2^{k-1}}+1\\ c_1,c_2> 1}}\Psi_{c_1c_2}(x)\prod_{\substack{d_2\mid 2^{2^k}+1\\ d_2>1}}\Psi_{d_2}(x),
\end{align*}
where we have used the trivial factorization $2^{2^k}-1=(2^{2^{k-1}}-1)(2^{2^{k-1}}+1)$ to deduce the last equality above. If $k=1$, then $\Phi_{t,f}(x)=\Psi_5(x)$, which is a quadratic polynomial. If $k=2$, then $\Phi_{t,f}(x)=\Psi_{15}(x)\Psi_{17}(x)$, where $\deg\Psi_{15}=4$ and $\deg\Psi_{17}=8$. It remains to consider $k>2$. Let $c_1,c_2>1$ be divisors of $2^{2^{k-1}}-1$ and $2^{2^{k-1}}+1$, respectively. Since $2^{2^{k-1}}-1$ and $2^{2^{k-1}}+1$ are coprime, so are $c_1$ and $c_2$. Let $q$ be a prime divisor of $c_2$. Then there exist $l,s\in \N$ such that $c_2=lq^s$ and $\gcd(l,q)=1$. Moreover, we have from Lemma~\ref{L:divisor}(i) that $q=2^{k+1}m+1$ for some $m\in \N$. By the remark under Theorem~\ref{T:phi} and multiplicativity of $\varphi$, we have
\[\deg \Psi_{c_1c_2}=\frac{\varphi(c_1c_2)}{2}=\frac{\varphi(c_1)\varphi(c_2)}{2}=\frac{\varphi(c_1)\varphi(l)q^{s-1}(q-1)}{2}=\varphi(c_1)\varphi(l)q^{s-1}2^k m,\]
so $t=2^k\mid \deg \Psi_{c_1c_2}$. It can be shown using the same argument that $2^k\mid \deg \Psi_{d_2}$ for any divisor $d_2>1$ of $2^{2^k}+1$.\\

{\it Case $p=3$.}

By Theorem~\ref{T:phi}, we have 
\begin{equation*}
\Phi_{t,f}(x)=\prod_{\substack{d_1\mid 2^{3^{k}}-1\\d_1\nmid 2^{3^{k-1}}-1}}\Psi_{d_1}(x)\prod_{\substack{d_2\mid 2^{3^k}+1\\d_2\nmid 2^{3^{k-1}}+1}}\Psi_{d_2}(x).
\end{equation*}
Let $d_1$ be a positive divisor of $2^{3^k}-1$, where $d_1\nmid 2^{3^{k-1}}-1$. 
Since \[2^{3^k}-1 = (2^{3^{k-1}}-1)((2^{3^{k-1}})^2+2^{3^{k-1}}+1),\]
where by \eqref{E:v3m}
\[\gcd(2^{3^{k-1}}-1,(2^{3^{k-1}})^2+2^{3^{k-1}}+1)=\gcd(2^{3^{k-1}}-1,2^{3^{k-1}}+2)=\gcd(2^{3^{k-1}}-1,3)=1,\]
there exists a prime divisor $q$ of $d_1$ such that $q \nmid 2^{3^{k-1}}-1.$ Let $s=v_q(d_1)$ and $l=d_1/q^s$. 
By Lemma~\ref{L:divisor}(ii), we have $q=2(3^km)+1$ for some $m\in\N$, whence
\[\deg \Psi_{d_1}=\frac{\varphi(d_1)}{2}=\frac{\varphi(l)q^{s-1}(q-1)}{2}=\varphi(l)q^{s-1}3^k m.\]
Now let $d_2$ be a positive divisor of $2^{3^k}+1$, where $d_2\nmid 2^{3^{k-1}}+1$. Observe that 
\[2^{3^k}+1 = (2^{3^{k-1}}+1)((2^{3^{k-1}})^2-2^{3^{k-1}}+1),\]
where 
\[\gcd(2^{3^{k-1}}+1,(2^{3^{k-1}})^2-2^{3^{k-1}}+1)=\gcd(2^{3^{k-1}}+1,2^{3^{k-1}}-2)=\gcd(2^{3^{k-1}}+1,3)=3.\]
If $d_2$ is a power of $3$, then we have from \eqref{E:v3p} that $d_2=3^{k+1}$, in which case
\[\deg\Psi_{d_2}=\frac{\varphi(d_2)}{2}=3^k.\]
Otherwise, $d_2$ has a prime divisor $q>3$ such that $q\nmid 2^{3^{k-1}}+1$, so we can again employ Lemma~\ref{L:divisor}(ii) to deduce that $3^k \mid \deg \Psi_{d_2}$.\\

{\it Case $p>3$.}

By Theorem~\ref{T:phi}, we have 
\begin{equation*}
\Phi_{t,f}(x)=\prod_{\substack{d_1\mid 2^{p^{k}}-1\\d_1\nmid 2^{p^{k-1}}-1}}\Psi_{d_1}(x)\prod_{\substack{d_2\mid 2^{p^k}+1\\d_2\nmid 2^{p^{k-1}}+1}}\Psi_{d_2}(x).
\end{equation*}
Simple calculations yield
\begin{equation*}
2^{p^k}-1 = (2^{p^{k-1}}-1)m_1,\quad 2^{p^k}+1 = (2^{p^{k-1}}+1)m_2,
\end{equation*}
where $m_1>1$, $m_2>1$, and $\gcd(2^{p^{k-1}}-1,m_1)=\gcd(2^{p^{k-1}}+1,m_2)=1$.
%Moreover, one sees from \eqref{E:v3p} and \eqref{E:v3m} that $3\nmid 2^{p^k}-1$ and $3\parallel 2^{p^k}+1$. 
Hence, for any $d_1$ and $d_2$ in the product above, there exist prime divisors $q_1$ and $q_2$ of $d_1$ and $d_2$ such that $q_1\nmid 2^{p^{k-1}}-1$, and $q_2\nmid 2^{p^{k-1}}+1.$ Then, with the aid of Lemma~\ref{L:divisor}(ii), it can be shown using arguments in the same vein as those in the previous cases that $p^k \mid \deg\Psi_{d_1}$ and $p^k \mid \deg\Psi_{d_2}.$
\end{proof}

We are now in a good position to prove Theorem~\ref{T:deg}.
\begin{proof}[Proof of Theorem~\ref{T:deg}]
% Therefore, it suffices to show that the degree of each irreducible factor of $\Phi_{n,f}(x)$ is divisible by $n$. 

The case $n=1$ is trivial, so we may assume that $n>1$. By the fundamental theorem of arithmetic, we can write $n$ as
$n= p_1^{a_1}\cdots p_r^{a_r},$
where $r,a_1,\ldots,a_r\in\N$ and $p_1,\ldots,p_r$ are distinct primes. For each $1\le i\le r$, let $n_i=n/p_i^{a_i}$ and $\beta_i=f^{(n_i)}(\alpha)\in \Q(\alpha)\subseteq K.$ Then it follows that $\beta_i$ is a periodic point of $f$ with minimal period $p_i^{a_i}.$ Recall that the periodic points of $f$ with minimal period $l$ in $\overline{\Q}$ are zeros of $\Phi_{l,f}(x)$. Hence $\beta_i$ must be a root of an irreducible factor $\tau_i(x)$ of $\Phi_{p_i^{a_i},f}(x)$. By Lemma~\ref{L:ppower}, we have, for every $1\le i \le r$,
\[ p_i^{a_i} \mid \deg \tau_i=[\Q(\beta_i):\Q].\]
Moreover, since $[\Q(\beta_i):\Q] \mid [K:\Q]$ (see the diagram below), it follows that $n=\lcm(p_1^{a_1},\ldots,p_r^{a_r}) \mid [K:\Q],$ as desired.
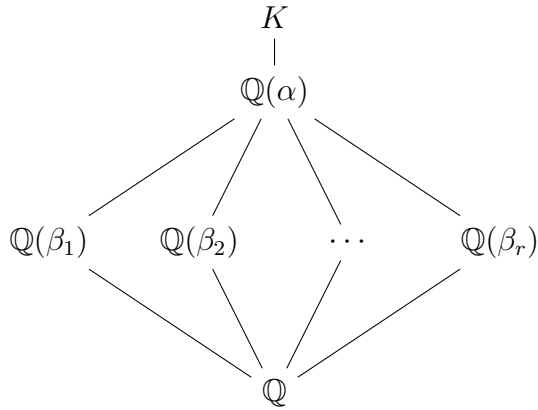
\begin{center}
\begin{tikzpicture}

    \node (Q1) at (0,0) {$\Q$};
    \node (Q2) at (-3,2) {$\Q(\beta_1)$};
    \node (Q3) at (-1,2) {$\Q(\beta_2)$};
    \node (Q4) at (1,2) {$\cdots$};
    \node (Q5) at (3,2) {$\Q(\beta_r)$};
    \node (Q6) at (0,4) {$\Q(\alpha)$};
    \node (Q7) at (0,5) {$K$};

    \draw (Q1)--(Q2);
    \draw (Q1)--(Q3);
    \draw (Q1)--(Q4);
    \draw (Q1)--(Q5);
    \draw (Q6)--(Q2);
    \draw (Q6)--(Q3);
    \draw (Q6)--(Q4);
    \draw (Q6)--(Q5);
    \draw (Q6)--(Q7);

    \end{tikzpicture}
\end{center}
\end{proof}
\section{Examples in number fields of low degree}\label{S:ex}
In this section, we apply our main results in classifying all values of cosine at a rational multiple of $\pi$ which belong to a number field $K$ of degree $1\le D\le 5$. We start by factoring the polynomial $f^{(D)}(x)-x$. By Theorem~\ref{T:deg}, the periodic points of $f$ in $K$ can be obtained from zeros of irreducible factors of $f^{(D)}(x)-x$ which have degree at most $D$. All the preperiodic points of $f$ in $K$ can then be computed by taking preimages of these values under $f$, which can be done in a finite number of steps.
\begin{example}
	For $D = 1$, we have $K = \mathbb{Q}$ and the periodic points must be fixed points of $f$. Since $f(x) - x = x^2 - x - 2 = (x + 1)(x - 2)$, we have the following digraphs representing all rational preperiodic points of $f$:
	\begin{center}
		\begin{tikzpicture}[->]
		
		\node (T2) at (0,0) {$-2$};
		\node (T1) at (-2,0) {$0$};
		\node (T3) at (2,0) {$2$};
		\node (B1) at (0,-1) {$1$};
		\node (B2) at (2,-1) {$-1$};

		\path (T1) edge        (T2);
		\path (T2) edge        (T3);
		\path (T3) edge [loop right]       (T3);
		\path (B1) edge        (B2);
		\path (B2) edge [loop right]       (B2);
		\end{tikzpicture}.
	\end{center}
Here $a\rightarrow b$ means $f(a)=b$. In other words, we have $\mathrm{PrePer}(f,\Q)=\{0, \pm 1, \pm 2\}$, which is equivalent to Niven's theorem.
\end{example}

\begin{example}
For $D = 2$, we have that $K$ is a quadratic number field; i.e., $K = \mathbb{Q}(\sqrt{m})$ for some square-free integer $m$. Note that $f^{(2)}(x) - x = x^4 - 4x^2 - x + 2 = (x - 2)(x + 1)(x^2 + x - 1)$. Hence the preperiodic points of $f$ in $K$ can be seen from the following digraphs:

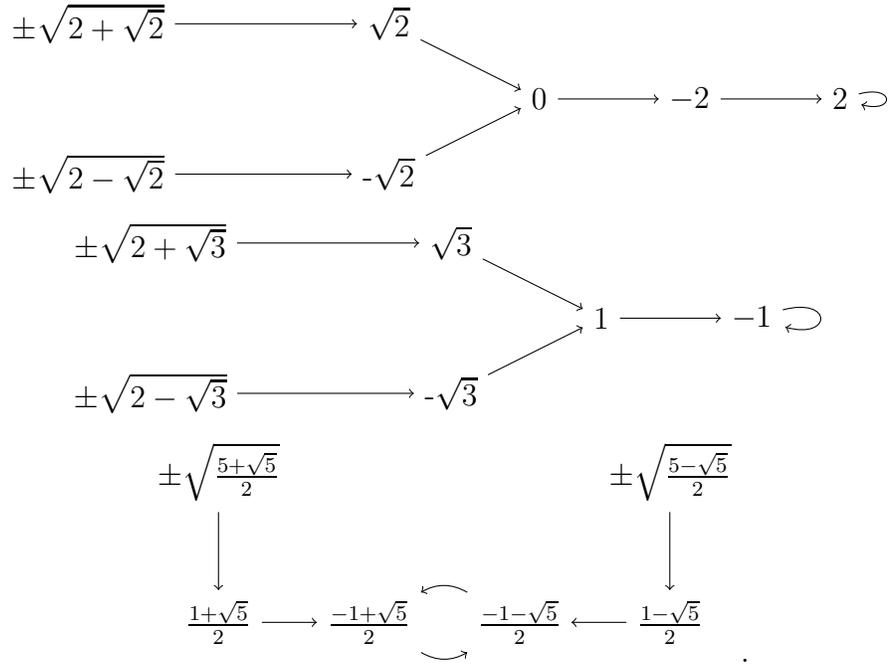
\begin{center}
	\begin{tikzpicture}[->]
	
	\node (T2) at (0,0) {$0$};
	\node (T1) at (-2,0) {$\pm \sqrt{2}$};
	\node (T3) at (2,0) {$-2$};
	\node (T4) at (4, 0)
	{$2$};
	\node (M1) at (0,-1) {$\pm \sqrt{3}$};
	\node (M2) at (2,-1) {$1$};
	\node (M3) at (4, -1) 
	{$-1$};
	\node (B1) at (-2, -2)
	{$\frac{1 + \sqrt{5}}{2}$};
	\node (B2) at (0, -2)
	{$\frac{-1 + \sqrt{5}}{2}$};
	\node (B3) at (2, -2)
	{$\frac{-1-\sqrt{5}}{2}$};
	\node (B4) at (4, -2)
	{$\frac{1-\sqrt{5}}{2}$};

	\path (T1) edge        (T2);
	\path (T2) edge        (T3);
	\path (T3) edge        (T4);
	\path (T4) edge [loop right]       (T4);
	\path (M1) edge        (M2);
	\path (M2) edge        (M3);
	\path (M3) edge [loop right]       (M3);
	\path (B1) edge        (B2);
	\path (B2) edge [bend right]       (B3);
	\path (B3) edge [bend right]       (B2);
	\path (B4) edge        (B3);
	\end{tikzpicture}.
\end{center}
This immediately gives Theorem~\ref{T:quadratic}.
\end{example}

\begin{example}
	For $D = 3$, we have $f^{(3)}(x) - x = (x - 2)(x + 1)(x^3 - 3x + 1)(x^3 + x^2 - 2x - 1)$. Suppose that $\alpha_1 < \alpha_2 < \alpha_3$ and $\beta_1 < \beta_2 < \beta_3$ are all roots of the third and the fourth factors respectively. Then we have the following digraphs representing all preperiodic points in cubic fields. 
	\begin{center}
		\begin{tikzpicture}[->]
		\node (A1) at (-1, 0) 
		{$-\alpha_2$};
		\node (A2) at (-1, 2.5) 
		{$-\alpha_1$};
		\node (A3) at (1.5, 2)
		{$-\alpha_3$};
		\node (B1) at (-1, 1)
		{$\alpha_1$};
		\node (B2) at (0, 2)
		{$\alpha_3$};
		\node (B3) at (1, 1)
		{$\alpha_2$};
		\node (B4) at (6, 2)
		{$0$};
		\node (B5) at (8, 2)
		{$-2$};
		\node (B6) at (10, 2)
		{$2$};
		\node (C1) at (-1, -4)
		{$-\beta_2$};
		\node (C2) at (-1, -1.5)
		{$-\beta_1$};
		\node (C3) at (1.5, -2)
		{$-\beta_3$};
		
		\node (D1) at (-1, -3)
		{$\beta_1$};
		\node (D2) at (0, -2)
		{$\beta_3$};
		\node (D3) at (1, -3)
		{$\beta_2$};
		\node (D4) at (8, -2)
		{$1$};
		\node (D5) at (10, -2)
		{$-1$};
		
		\path (A1) edge [bend left]       (B1);
		\path (A2) edge [bend left]		(B2);
		\path (A3) edge	[bend left]		(B3);
		\path (B1) edge	[bend left]		(B2);
		\path (B2) edge [bend left]		(B3);
		\path (B3) edge [bend left]        (B1);
		\path (B4) edge 		(B5);
		\path (B5) edge			(B6);
		\path (B6) edge	[loop right]		(B6);
		\path (C1) edge	[bend left]		(D1);
		\path (C2) edge	[bend left]		(D2);
		\path (C3) edge	[bend left]		(D3);
		\path (D1) edge	[bend left]	(D2);
		\path (D2) edge 	[bend left]	(D3);
		\path (D3) edge [bend left]        (D1);
		\path (D4) edge			(D5);
		\path (D5) edge [loop right]       (D5);
		\end{tikzpicture}
	\end{center}
Therefore, $\{\pm \alpha_1, \pm \alpha_2, \pm \alpha_3, \pm \beta_1, \pm \beta_2, \pm \beta_3\}$ is the set of all cubic irrational values of $2\cos (r \pi),$ where $r\in \Q$.
\end{example}

\begin{example}
For $D = 4$, we have $f^{(4)}(x) - x = (x - 2)(x + 1)(x^2 + x - 1)(x^4 - x^3 - 4x^2 + 4x + 1)(x^8 + x^7 - 7x^6 - 6x^5 + 15x^4 + 10x^3 - 10x^2 - 4x + 1)$.  Let $\alpha_1<\alpha_2<\alpha_3<\alpha_4$ be the roots of the quartic factor of $f^{(4)}(x) - x$. Then we have the following digraphs:

\begin{center}
	\begin{tikzpicture}[->]
	\node (AA1) at (-1.5, 5.5) {-$\alpha_1$};
	\node (A1) at (-2, 4) {$\alpha_1$};
	\node (A2) at (0, 4) {$\alpha_4$};
	\node (B1) at (-3.5, 2.5) {$-\alpha_2$};
	\node (B2) at (1.5, 3.5) {$-\alpha_4$};
	\node (C1) at (-2, 2) {$\alpha_2$};
	\node (C2) at (0, 2) {$\alpha_3$};
	\node (D1) at (-.5, .5) {-$\alpha_3$};
	
	\path (AA1) edge [bend left] 		(A2);
	\path (A1) edge [bend left]			(A2);
	\path (A2) edge [bend left]			(C2);
	\path (C2) edge [bend left]			(C1);
	\path (C1) edge [bend left]			(A1);
	\path (B1) edge [bend left]			(A1);
	\path (B2) edge [bend left]			(C2);
	\path (D1) edge [bend left]			(C1);
	\end{tikzpicture}
\end{center}

\begin{center}
	\begin{tikzpicture}[->]
	\node (A1) at (-6, 1) {$\pm \sqrt{2 + \sqrt{2}}$};
	\node (A2) at (-2, 1) {$\sqrt{2}$};
	\node (B1) at (0, 0) {$0$};
	\node (B2) at (2, 0) {$-2$};
	\node (B3) at (4, 0) {$2$};
	\node (C1) at (-6, -1) {$\pm \sqrt{2 - \sqrt{2}}$};
	\node (C2) at (-2, -1) {-$\sqrt{2}$};
	
	\path (A1) edge (A2) ;
	\path (A2) edge (B1) ;
	\path (B1) edge (B2) ;
	\path (B2) edge (B3) ;
	\path (B3) edge [loop right] (B3);
	\path (C1) edge (C2) ;
	\path (C2) edge (B1);
	\end{tikzpicture}
\end{center}	

\begin{center}
	\begin{tikzpicture}[->]
	\node (A1) at (-6, 1) {$\pm \sqrt{2 + \sqrt{3}}$};
	\node (A2) at (-2, 1) {$\sqrt{3}$};
	\node (B1) at (0, 0) {$1$};
	\node (B2) at (2, 0) {$-1$};
	\node (C1) at (-6, -1) {$\pm \sqrt{2 - \sqrt{3}}$};
	\node (C2) at (-2, -1) {-$\sqrt{3}$};
	
	\path (A1) edge (A2) ;
	\path (A2) edge (B1) ;
	\path (B1) edge (B2) ;
	\path (B2) edge [loop right] (B2);
	\path (C1) edge (C2) ;
	\path (C2) edge (B1);
	\end{tikzpicture}
\end{center}

\begin{center}
	\begin{tikzpicture}[->]
	\node (A1) at (-2, 2) {$\pm \sqrt{\frac{5 + \sqrt{5}}{2}}$};
	\node (A2) at (4, 2) {$\pm \sqrt{\frac{5 - \sqrt{5}}{2}}$};
	\node (B1) at (-2, 0) {$\frac{1 + \sqrt{5}}{2}$};
	\node (B2) at (0, 0) {$\frac{-1 + \sqrt{5}}{2}$};
	\node (B3) at (2, 0) {$\frac{-1 - \sqrt{5}}{2}$};
	\node (B4) at (4, 0) {$\frac{1 - \sqrt{5}}{2}$};
	
	\path (A1) edge (B1);
	\path (A2) edge (B4);
	\path (B1) edge (B2);
	\path (B2) edge [bend right] (B3);
	\path (B3) edge [bend right] (B2);
	\path (B4) edge (B3);
	\end{tikzpicture}.
\end{center}	
Therefore, $\left\{ \pm \alpha_1, \pm \alpha_2, \pm \alpha_3, \pm \alpha_4, \pm \sqrt{2 \pm \sqrt{2}}, \pm \sqrt{2 \pm \sqrt{3}}, \pm \sqrt{\frac{5 \pm \sqrt{5}}{2}} \right\}$ is the set of all quartic irrational values of $2\cos (r \pi)$, where $r \in \mathbb{Q}$.
\end{example}

\begin{example}
For $D = 5$, we can factorize 
$$f^{(5)}(x) - x = (x - 2)(x + 1)(x^5 + x^4 - 4x^3 - 3x^2 + 3x + 1)\cdot g(x) \cdot h(x),$$
 where $g(x)$ and $h(x)$ are irreducible polynomials with $\deg g = 10$ and $\deg h = 15$. Let $\alpha_1<\alpha_2<\alpha_3<\alpha_4<\alpha_5$ be the roots of the quintic factor of $f^{(5)}(x) - x$. Then we have the following digraphs:
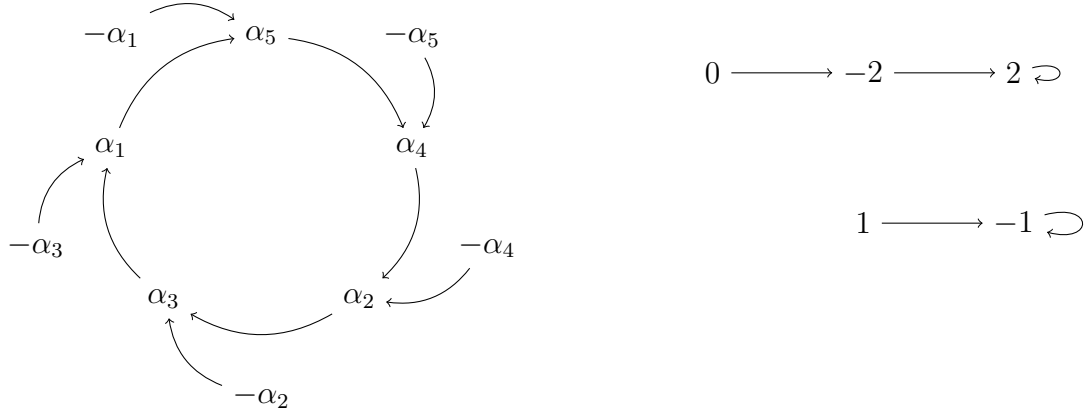
\begin{center}
	\begin{tikzpicture}[->]
	\node (A1) at (-2, 1.5) {$-\alpha_1$};
	\node (B1) at (0, 1.5) {$\alpha_5$};
	\node (B2) at (2, 1.5) {$-\alpha_5$};
	\node (C1) at (-2, 0) {$\alpha_1$};
	\node (C2) at (2, 0) {$\alpha_4$};
	\node (C3) at (3, -1.3) {$-\alpha_4$};
	\node (D1) at (-3, -1.3) {$-\alpha_3$};
	\node (D2) at (-1.3, -2) {$\alpha_3$};
	\node (D3) at (1.3, -2) {$\alpha_2$};
	\node (E1) at (0, -3.3) {$-\alpha_2$};
	
	\node (B3) at 	(6, 1) {$0$};
	\node (B4) at   (8, 1) {$-2$};
	\node (B5) at   (10, 1) {$2$};
	\node (C4) at   (8, -1)  {$1$};
	\node (C5) at   (10, -1) {$-1$};
	
	\path (A1) edge [bend left] (B1);
	\path (B1) edge [bend left] (C2);
	\path (C2) edge [bend left] (D3);
	\path (D3) edge [bend left] (D2);
	\path (D2) edge [bend left] (C1);
	\path (C1) edge [bend left] (B1);
	\path (B2) edge [bend left] (C2);
	\path (C3) edge [bend left] (D3);
	\path (E1) edge [bend left] (D2);
	\path (D1) edge [bend left] (C1);
	
	\path (B3) edge  (B4);
	\path (B4) edge  (B5);
	\path (B5) edge [loop right]  (B5);
	\path (C4) edge  (C5);
	\path (C5) edge [loop right]  (C5);
	\end{tikzpicture}.
\end{center}	
Therefore, $\{ \pm \alpha_1, \pm \alpha_2, \pm \alpha_3, \pm \alpha_4, \pm \alpha_5 \}$ is the set of all quintic irrational values of $2\cos (r \pi),$ where $r \in \mathbb{Q}$. 
\end{example}

\section{Closing remarks}\label{S:rem}
The dynamical properties of the map $f_c(x)=x^2+c$ have been studied extensively over the past few decades, yet many related problems still remain open. For $c=-2$, this map turns out to be closely related to a classical result in number theory, namely Niven's theorem, and its extensions. This relation can be seen directly from Theorem~\ref{T:CK}. We can then apply Theorem~\ref{T:deg} to systematically classify all preperiodic points of $f_{-2}(x)$ in any number field $K$. It should be remarked that our proof of Theorem~\ref{T:deg} relies crucially on the known result \cite{VH} about factorization of the dynatomic polynomials associated to $f_{-2}(x)$, so it should not be expected that this theorem holds for $f_{c}(x)$ in general. As a concrete example, consider $f_{-1}(x)=x^2-1$. It is clear that $0$ is a periodic point of $f_{-1}(x)$ with minimal period $2$, so $0$ is not a fixed point of $f_{-1}(x)$. To determine all values of $c\in \mathbb{Q}$ for which $f_{c}(x)$ satisfies the property in Theorem~\ref{T:deg}, one might start from those in \cite[Fig.~1]{Poonen} which correspond to digraphs containing no cycles of length greater than $1$; i.e., $c\in \{1,1/4,0,-2,-3/4,-10/9\}$. For each of these values, it is also an interesting problem to interpret the preperiodic points of $f_{c}(x)$ as special values of some function.

\section*{Acknowledgements}
The second author is supported by National Research Council of Thailand (NRCT) under the Research Grant for Mid-Career Scholar no. N41A640153.
The third author acknowledges funding by Office of the Permanent Secretary, Ministry of Higher Education, Science, Research
and Innovation (OPS MHESI), Thailand Science Research and Innovation (TSRI) and King Mongkut's University of Technology Thonburi (Grant No. RGNS 64-096).

\bibliographystyle{amsplain}
\bibliography{ref}
\end{document}